\documentclass[12pt,letterpaper,reqno]{amsart}
\usepackage{amsmath,amsthm,amsfonts,amssymb,mathtools}

\theoremstyle{plain}
\newtheorem{theorem}{Theorem}
\newtheorem{lemma}{Lemma}

\theoremstyle{definition}
\newtheorem{definition}{Definition}

\newcommand{\IL}{\ensuremath{\mathsf{IL}}}
\newcommand{\GL}{\ensuremath{\mathsf{GL}}}
\newcommand{\ILF}{\ensuremath{\mathsf{ILF}}}
\newcommand{\ILC}{\ensuremath{\mathsf{IL_0}}}
\newcommand{\share}[1]{\mu(\hbox{#1})}
\newcommand{\param}[1]{\lambda(\hbox{#1})}

\begin{document}

\title[On the share of closed \IL\ formulas which are also in \GL]{\hbox{On the share of closed \IL\ formulas which are also in \GL}}

\author{Vedran \v{C}a\v{c}i\'{c}}
\address{Vedran \v{C}a\v{c}i\'{c}, University of Zagreb, Department of Mathematics, Bijeni\v{c}ka cesta 30, 10000 Zagreb, Croatia}
\email{veky@math.hr}

\author{Vjekoslav Kova\v{c}}
\address{Vjekoslav Kova\v{c}, University of Zagreb, Department of Mathematics, Bijeni\v{c}ka cesta 30, 10000 Zagreb, Croatia}
\email{vjekovac@math.hr}

\begin{abstract}
Normal forms for wide classes of closed \IL\ formulas were given in \cite{MC}. Here we quantify asymptotically, in exact numbers, how wide those classes are. As a consequence, we show that the ``majority'' of closed \IL\ formulas have \GL-equivalents, and by that, they have the same normal forms as \GL\ formulas. Our approach is entirely syntactical, except for applying the results of \cite{MC}. As a byproduct we devise a convenient way of computing asymptotic behaviors of somewhat general classes of formulas given by their grammar rules. Its applications do not require any knowledge of the recurrence relations, generating functions, or the asymptotic enumeration methods, as all these are incorporated into two fundamental parameters.
\end{abstract}


\maketitle

\section{Motivation}
G\"{o}del's incompleteness theorems were a big breakthrough for mathematical logic. With time, logicians started to wonder how they can be generalized, and what else, based on some simple facts we knew, could be deduced about provability predicates. Formalizing provability over some base theory $T$ as a unary modal operator $\Box$ led to the theory \GL\ (named after G\"{o}del and L\"{o}b), which we know today as the provability logic of many base theories.

Provability is a valuable tool for judging absolute strength of some formula against a theory, but what about its relative strength? For some base theory $T$ and two formulas $F$ and $G$ we can ask if $T+F$ is interpretable in $T+G$. In other words: can we find a way of reinterpreting symbols of $T$, and restricting quantifiers using some definable predicate, preserving provability of whole $T$, but such that (reinterpreted) formula $F$ becomes a theorem if we add $G$ as an axiom? Here we do not just divide formulas into black and white, but try to order them in various shades of gray. In fact, various colors would be a better analogy, since the ordering is usually not total.

We can obtain logics of interpretability in a manner quite analogous to provability logic. Formalizing in\-ter\-pre\-ta\-bi\-li\-ty in the above sense as a binary modal operator $\rhd$, we are led to various interpretability logics, most basic of which is probably \IL. Many detailed definitions and explanations concerning \IL\ can be found in~\cite{overview}. Unfortunately, \IL\ itself is, unlike \GL, just a ``nice minimal fragment'' of those interpretability logics, and different base theories add to \IL\ different principles of interpretability, extending it in diverse ways. The intersection of interpretability logics of various ``reasonable'' arithmetical theories is still searched for; see~\cite{Goris} for some results.

However, we can still consider properties of \GL\ and ask ourselves if \IL\ has something analogous. One well-known property of \GL\ is that its closed formulas have very regular normal forms: Every \GL\ formula without variables is equivalent to a Boolean combination of formulas $\bot$, $\Box\bot$, $\Box\Box\bot$, and so on. That Boolean combination can be further normalized, taking into account that $\Box^n\bot\rightarrow\Box^m\bot$ whenever $n\le m$. See~\cite{GLtrace} for details.

Do \IL\ formulas have something similar? Even before this question, we are faced with another one: What are the basic blocks in the closed fragment of \IL? In \GL\ it was easy --- repeating $\Box$ before $\bot$ gives a natural single-parameter countable family acting as ``propositional variables'' and these terms are to be connected into Boolean combinations. The only thing analogous in \IL\ is that same family. Namely, it is easily seen that $\Box$ can be emulated in \IL, $\Box A$ being equivalent to $\neg A\rhd\bot$. Note that $\bot$ is invariant under interpretation and $T+\neg A$ is inconsistent iff $T\vdash A$. Thus the same family $\{\Box^n\bot\,:\,n\in\mathbb N\}$ is also available in \IL.

Some strengthenings of \IL\ actually have the property that all their closed formulas have \GL-equivalents, which particularly implies that they have the same normal forms as \GL\ formulas. For example, that property is proved for any logic containing \ILF\ in~\cite{ILFnormalform}. However, this is \emph{not} the case with plain \IL. In \cite{MC} we showed that many, but not all, \IL\ formulas have \GL-equivalents. Here we count those formulas and find their exact share in the whole closed fragment of \IL. An interesting but probably very difficult open problem would be to classify all the exceptional formulas. Justification of our heavily computational approach lies precisely in a desire to at least quantify their portion, and to see to what extent \IL\ ``comes close'' to \GL.

\section{Introduction}
\label{sectionintro}
The formulas of the closed fragment of \IL\ (which we will hereinafter call \ILC), due to their inductive formation, naturally form a structural hierarchy with respect to their \emph{complexity}, i.e.\@ the number of logical connectives. That complexity certainly depends on a set of symbols we decide to use. One possibility is to take a minimal set of symbols, consisting only of $\bot$, $\rightarrow$, and $\rhd$. It is not hard to see that these are sufficient: Negation of $A$ can be expressed as $A\rightarrow\bot$; it is well known that all Boolean connectives can be expressed with $\rightarrow$ and $\neg$; we have already discussed $\Box$; and $\Diamond$ can be expressed with $\Box$ and $\neg$ in the usual way. However, many other choices of the alphabet are possible, such as those including $\Box$ or $\top$.

In order to make the combinatorial part of this paper generally applicable, we leave \IL\ for a moment and work in the following setting. Suppose that we are given $p$ unary connectives $\diamond_1,\ldots,\diamond_p$, $q$ binary connectives $\ast_1,\ldots,\ast_q$, and $s$ primitive symbols $\clubsuit_1,\ldots,\clubsuit_s$, where $p\geq 0$, $q\geq 1$, $s\geq 1$. We consider closed formulas on this alphabet and let $\mathcal{F}_n$ denote the set of all such formulas of complexity $n$, that is, with $n$ connectives $\diamond_i$ and $\ast_j$ in total. For any property $\mathcal{P}$ of closed formulas we can define
\begin{equation}\label{eqsharedef}
\mu(\mathcal{P}) := \lim_{n\to\infty}\frac{\#\{F\in\mathcal{F}_n\,:\,\mathcal{P}(F)\}}{\#\mathcal{F}_n}.
\end{equation}
Here $\#$ denotes the number of elements in a finite set. The quantity $\mu(\mathcal{P})$ could be called the \emph{share} or the \emph{asymptotic density} of formulas with property $\mathcal{P}$. Note that the limit in \eqref{eqsharedef} does not have to exist. For example, formulas of even complexity do not have a share, as the corresponding sequence in \eqref{eqsharedef} alternates: $1,0,1,0,1,\ldots$. However, properties defined ``independently'' of complexity usually have a share. Also note that if only finitely many formulas have property $\mathcal{P}$, then $\mu(\mathcal{P})=0$, but the converse does not hold.

There is an auxiliary quantity associated with $\mathcal{P}$, which is closely linked to $\mu(\mathcal{P})$, but does not have an immediately obvious interpretation. Let us first denote
\begin{equation}\label{eqexpressionforr}
r = \frac{1}{p + 2qs + 2\sqrt{(p+qs)qs}}.
\end{equation}
This number captures the rate of the exponential growth of the total number of formulas with complexity $n$. More precisely, in Lemma \ref{lemma01}(d) we will show $\lim_{n\to\infty}\limits \#\mathcal{F}_{n+1} / \#\mathcal{F}_n = 1/r$. Next, for a property $\mathcal{P}$ we can define
$$ \lambda(\mathcal{P}) := \sum_{n=0}^{\infty} r^{n+1} \#\{F\in\mathcal{F}_n\,:\,\mathcal{P}(F)\}. $$
As a consequence of Lemma \ref{lemma01}(a) we will conclude that $\lambda(\mathcal{P})$ always exists as a finite nonnegative number. Its vague interpretation is that it constitutes an exponentially weighted total number of formulas with property $\mathcal{P}$, but the true meaning will be revealed in Sections \ref{sectioncombinatorics} and \ref{sectionproofs}. Our motivation for its definition is simply that $\mu(\mathcal{P})$ will often be expressed in terms of $\lambda(\mathcal{P})$. Obviously, $\lambda(\mathcal{P})=0$ if and only if there are no formulas satisfying property $\mathcal{P}$.

If some uppercase letter, such as $W$, denotes a generic formula with property $\mathcal{P}$, then we will also simply write $\mu_W$ and $\lambda_W$ in place of $\mu(\mathcal{P})$ and $\lambda(\mathcal{P})$. Let us also agree to use the corresponding calligraphic letter indexed by $n$, such as $\mathcal{W}_n$, to denote the set of those formulas of complexity $n$.

If the letter $F$ stands for an arbitrary closed formula, then we can schematically write the production rules as
$$ F \,::=\ \clubsuit_1 \ | \,\cdots\, | \ \clubsuit_s \ | \ \diamond_1\! F \ | \,\cdots\, | \ \diamond_p\! F \ | \ F \ast_1\! F \ | \,\cdots\, | \ F \ast_q\! F. $$
Each occurrence of the letter $F$ on the right hand side can represent a different closed formula, having strictly smaller complexity than the one on the left hand side. We are interested in calculating shares of certain subclasses of formulas given by similar production rules.

Let us now state our main combinatorial result in such a way that it can even be used as a ``black box'' for quantifying various classes of formulas, such as those appearing in \IL.

\begin{theorem}\label{theorem01}
\begin{itemize}
\item[(a)]
Suppose that a class of formulas $W$ is given by the productions
$$ W \,::=\, U \ | \ V, $$
where $U$ and $V$ represent two classes of formulas such that $\mathcal{U}_n$ and $\mathcal{V}_n$ are disjoint for each $n$. If $\mu_U$ and $\mu_V$ both exist, then $\mu_W$ also exists and we have
$$ \lambda_W = \lambda_U + \lambda_V,\quad \mu_W = \mu_U + \mu_V. $$
\item[(b)]
Suppose that a class of formulas $W$ is given by the productions
\begin{align*} W \,::=\ U \ | \ \diamond_1\! W \ | \,\cdots\, | \ \diamond_{p'}\! W \ | \ W \ast_1\! W \ | \,\cdots\, | \ W \ast_{q'}\! W \ | & \\
| \ (V\setminus W) \ast_1\! W \ | \,\cdots\, | \ (V\setminus W) \ast_{q''}\! W. &
\end{align*}
Here $p',q',q''$ are integer parameters such that $0\leq p'\leq p$, $0\leq q'\leq q$, $0\leq q''\leq q$, and $(p',q',q'')\neq(p,q,0)$, while $U$ and $V$ represent two classes of formulas. We additionally require that $\mathcal{U}_n$ does not contain any other formulas of complexity $n$ appearing on the right hand side, that $\mathcal{U}_0$ contains at least one primitive symbol, and that $\mathcal{W}_n$ is a subset of $\mathcal{V}_n$. If $\mu_U$ and $\mu_V$ both exist, then $\mu_W$ also exists and we have
\begin{align*}\lambda_W & = \frac{2\lambda_U}{1-p'r-q''\lambda_V + \sqrt{(1-p'r-q''\lambda_V)^2 - 4(q'-q'')\lambda_U}}, \\
\mu_W & = \frac{\mu_U + q''\lambda_W\mu_V}{1-p'r-q''\lambda_V - 2(q'-q'')\lambda_W}.
\end{align*}
\item[(c)]
Suppose that a class of formulas $W$ is given by the productions
$$ W \,::=\ \clubsuit_1 \ | \,\cdots\, | \ \clubsuit_{s'} \ | \ \diamond_1\! T \ | \,\cdots\, | \ \diamond_{p'}\! T \ | \ U \ast_1\! V \ | \,\cdots\, | \ U \ast_{q'}\! V. $$
Here $p',q',s'$ are integer parameters such that $0\leq p'\leq p$, $0\leq q'\leq q$, and $0\leq s'\leq s$, while $T$, $U$, $V$ represent three classes of formulas. If $\mu_T$, $\mu_U$, and $\mu_V$ all exist, then $\mu_W$ also exists and we have
$$ \lambda_W = rs' + p'r\lambda_T + q'\lambda_U \lambda_V,\quad \mu_W = p'r\mu_T + q'\lambda_V \mu_U + q'\lambda_U \mu_V. $$
\item[(d)]
Suppose that a class of formulas $W$ is given by the productions
$$ W \,::=\ U \ | \ \diamond_1\! W \ | \,\cdots\, | \ \diamond_{p'}\! W \ | \ \clubsuit_1 \ast_1\! (\diamond_1 W) \ | \,\cdots\, | \ \clubsuit_{q'} \ast_{q'}\! (\diamond_{q'} W). $$
Here $p'$ and $q'$ are integer parameters such that $0\leq p'\leq p$ and $0\leq q'\leq q$, while $U$ represents some class of formulas such that $\mathcal{U}_n$ does not contain any other formulas of complexity $n$ appearing on the right hand side. If $\mu_U$ exists, then $\mu_W$ also exists and we have
$$ \lambda_W = \frac{\lambda_U}{1-p'r-q'r^2},\quad \mu_W = \frac{\mu_U}{1-p'r-q'r^2}. $$
\end{itemize}
\end{theorem}

It is worthwhile noting that one is allowed to alter the productions in the above theorem in ways that do not change the cardinalities of the corresponding sets of formulas with complexity $n$, and keep the sets of formulas coming from different productions disjoint. For instance, $(V\setminus W) \ast_1\! W$ in Part (b) can be replaced with its ``reflected'' counterpart $W \ast_1\! (V\setminus W)$. Also, $\clubsuit_1 \ast_1\! (\diamond_1 W)$ can be replaced with $\clubsuit_1 \ast_1\! (\clubsuit_1 \ast_1\! W)$, $\diamond_1 (\clubsuit_1 \ast_1\! W)$, or $\diamond_1 (\diamond_1 W)$, as long as it never coincides with any of $\diamond_i W$ for $i=1,\ldots,p'$. Finally, the unary connectives $\diamond_i$ can be interchanged arbitrarily throughout the theorem, and the same is true for the binary connectives $\ast_j$, or the primitive symbols $\clubsuit_k$.

The proof of Theorem~\ref{theorem01} will span over the Sections \ref{sectioncombinatorics} and \ref{sectionproofs}. Let us remark that Part (b) will be the most difficult result to prove, while Part (a) is an immediate consequence of the definitions. Even though there already exists a vast amount of combinatorial literature on the enumeration techniques, in Section~\ref{sectionproofs} we come up with a simple trick that allows us to compare asymptotic behaviors of certain recursively defined sequences. Our technique is a more practical alternative to the well-known method of computing sequence asymptotics from the expression for its generating function, which can be found in \cite[Chapters~IV--VIII]{FS}. Due to multiple-recursive nature of the definitions to follow, the latter method would be computationally very hard to apply, leading to multiple-page expressions for generating functions. The only exception is the total class of formulas $F$, where we will resort to that heavy machinery just to derive the asymptotics of $\#\mathcal{F}_n$. However, we believe that the main novelty of Theorem~\ref{theorem01} lies in its elementary formulation and ready applicability to counting problems.

\smallskip
In Section~\ref{sectiondefinitions} we will define certain classes of formulas to which the results of \cite{MC} can be applied and in Section~\ref{sectioncalculations} we will repeatedly apply Theorem~\ref{theorem01} to evaluate their shares. This will result in a lower bound on the share of formulas in \ILC\ that are equivalent to a formula in \GL. The exact numerical value of this bound depends on the set of symbols we agree to work with. We work with three reasonable alphabet choices in order to argue that the ``\GL-like'' formulas we consider constitute a ``majority'' of \IL.

\begin{theorem}\label{theorem02}
\begin{itemize}
\item[(a)] If the alphabet is $\{\bot,\rightarrow,\rhd\}$, then $\share{\GL-like} = 0.93771\ldots$.
\item[(b)] If the alphabet is $\{\bot,\rightarrow,\rhd,\Box\}$, then $\share{\GL-like} = 0.86653\ldots$.
\item[(c)] If the alphabet is $\{\bot,\top,\rightarrow,\rhd,\Box\}$, then $\share{\GL-like} = 0.90519\ldots$.
\end{itemize}
\end{theorem}

Intuitively, according to Part (a), if we pick an \ILC\ formula of large enough complexity at random, there is less than $7\%$ chance that it \emph{does not} have a normal form as some closed \GL\ formula. Of course, the true probability can be much smaller. What we will call ``\GL-like formulas'' will be just the ones that are syntactically simple enough that they can be easily seen to have \GL-equivalents. But even so, the result is quite surprising in our opinion, especially because it is obtained by an almost entirely syntactical approach. Let us also remark that Theorem~\ref{theorem01} will guarantee the existence of the limits in Theorem~\ref{theorem02}. Their exact values are quite complicated numerical expressions; see Section~\ref{sectioncalculations}.

The fact that the limit $\share{$\GL$-like}$ is so close to $1$ is also interesting from a complexity point of view: In~\cite{Bou} it is proved that establishing whether an \ILC\ formula is a theorem is PSPACE-hard. This means that even if there were a normal form theorem for the \ILC, it would be practically useless, since it must be very difficult to find a normal form for an arbitrary \IL\ formula. However, this article shows that ``practicality'' can be understood differently: the great majority of instances of theoremhood problem are actually very easy and solvable in polynomial time. A small number of them are very hard, although they are easily recognizable.

\section{Defining our classes}
\label{sectiondefinitions}
Let us review the main result of \cite{MC}. First we will state the definitions in the equivalent form that fits better in the context of this article.

\begin{definition}
In \ILC\ with a minimal syntax, the classes of all formulas ($F$), \emph{affirmative formulas} ($A$), \emph{negative formulas} ($N$), \emph{direct formulas} ($M$) and \emph{cross formulas} ($X$) are given by the following productions:
$$ \begin{array}{l}
    F::=A\ | \ N,\\
    A::=F\rhd F \ | \ N\rightarrow F \ | \ F\rightarrow A,\\
    N::= \bot \ | \ A\rightarrow N,\\
    M::= A\rhd N,\\
    X::= M\rightarrow\bot.
\end{array}$$
In words, any \ILC\ formula is either affirmative or negative. Each negative formula is either $\bot$, or is of the form $A\rightarrow N'$, where $A$ is an affirmative formula and $N'$ is a smaller negative formula; and so on.
\end{definition}

\begin{theorem}[from \cite{MC}]\label{theoremmc}
Let $F$ be an \ILC\ formula, $A$ an affirmative formula, $N$ a negative formula, $M$ a direct formula, and $X$ a cross formula. Then we have:
$$\begin{array}{r@{\,\rhd\,}l@{\;\Longleftrightarrow\;}l}
\bot&F&\top,\\
F&A&\top,\\
A&\bot&\Box\bot,\\
N&X&\top,\\
X&\bot&\Box\Box\bot,
\end{array}$$
$$\begin{array}{r@{\;\stackrel{globally}{\Longleftrightarrow}\;}l}
N&\bot,\\
M&\Box\bot.
\end{array}$$
\end{theorem}

Next, we say that an \ILC\ formula is \emph{basic locally \GL-like} if it belongs to at least one of the five classes of formulas having locally equivalent \GL\ formulas according to Theorem~\ref{theoremmc}, namely
$$ \bot\rhd F,\quad F\rhd A,\quad A\rhd\bot,\quad N\rhd X,\quad X\rhd\bot . $$
We say that an \ILC\ formula is \emph{locally \GL-like} if it is $\bot$, or it is basic locally \GL-like, or it is a Boolean combination $L'\rightarrow L''$ of smaller locally \GL-like formulas $L'$ and $L''$. We denote basic locally \GL-like formulas by the letter $B$ and locally \GL-like formulas by the letter $L$. Obviously, all locally \GL-like formulas, not just the basic ones, have normal forms in \GL: If $F$ has normal form $F_1$ and $F'$ has normal form $F'_1$, then $F\rightarrow F'$ has normal form equivalent to $F_1\rightarrow F'_1$.

We have not yet taken negative and direct formulas into account, as they only possess global \GL-equivalents. Working with non-locally \GL-like formulas is harder, since we do not have the substitution property we had for locally \GL-like formulas: A Boolean combination of two formulas having global \GL-equivalents does not necessarily have a global \GL-equivalent. If it were true, then since every affirmative formula is equivalent to negation of a negative formula, we would have global \GL-equivalents for all \ILC\ formulas, which is known not to be the case; see~\cite[Section 3]{MC}.

However, there is a way to get larger \GL-like formulas from smaller ones, and that is the $\Box$ operator from \GL, or more precisely, its \IL-equivalent. It is easily seen that $(F\rightarrow\bot)\rhd\bot$ is locally equivalent to $\Box F$, and if $F$ has a global \GL-equivalent $F_1$, then $(F\rightarrow\bot)\rhd\bot$ has a global \GL-equivalent (that is equivalent to) $\Box F_1$. This motivates the following definition.

\begin{definition}
We say that an \ILC\ formula is \emph{\GL-like} if it is locally \GL-like, or negative, or direct, or of the form $(G'\rightarrow\bot)\rhd\bot$, where $G'$ is a smaller \GL-like formula. We denote \GL-like formulas by the letter $G$.
\end{definition}

Our ultimate goal is to compute the share of \GL-like formulas. To achieve this we will obviously need to progress gradually through the simpler classes of formulas defined above. Before attempting to evaluate shares of these classes, we insert two sections devoted to establishing asymptotic behaviors of some recursively defined sequences.

\section{Statements of the results on sequence asymptotics}
\label{sectioncombinatorics}

In this section we state three auxiliary lemmas that constitute a purely combinatorial reformulation of Theorem~\ref{theorem01}. Their proofs will be postponed to the next section. This material can be skipped without hurting the understanding of the ``logical'' part of the article.

Let $p\geq 0$, $q\geq 1$, and $s\geq 1$ be three integer parameters; they can be interpreted as in Section~\ref{sectionintro}. Let $f_n=\#\mathcal{F}_n$ denote the total number of closed formulas of complexity $n\geq 0$. Observe that the only formulas of complexity $0$ are the primitive symbols, so $f_0=s$. Moreover, each formula of complexity $n\geq 1$ is composed either by applying the main unary connective to a formula of complexity $n-1$, or is divided by its main binary connective into two formulas whose complexities add up to $n-1$. That way we arrive at the recurrence relation that enumerates the total number of closed formulas,
\begin{equation}\label{eqbasicrecurrence}
f_0=s, \qquad f_{n} = p f_{n-1} + q \sum_{k=0}^{n-1} f_k f_{n-k-1} \quad\textrm{for } n\geq 1.
\end{equation}
For some choices of the parameters the obtained sequence $(f_n)_{n=0}^{\infty}$ is standard and named. Take $s=1$ for simplicity. If $p=0$ and $q$ is arbitrary, then $f_n=q^n C_n$, where $(C_n)_{n=0}^{\infty}$ stands for the \emph{Catalan numbers}, see entry A000108 in \cite{OEIS}. For $p=1$, $q=1$ the sequence $(f_n)_{n=0}^{\infty}$ becomes the \emph{Large Schr\"{o}der numbers}, entry A006318, for $p=1$, $q=2$ we obtain sequence A103210, while for $p=2$, $q=1$ we get sequence A047891 in the encyclopedia \cite{OEIS}. Let us remark that a closed expression for $f_n$ does not seem to exist in the cases when $p\geq 1$.

We will need some of the standard asymptotic enumeration techniques to extract the desired growth properties of $f_n$. Consider the ordinary generating function of the shifted sequence $(f_{n-1})_{n=1}^{\infty}$, which is given by the formula
$$ f(z) = \sum_{n=1}^{\infty} f_{n-1} z^n = \sum_{n=0}^{\infty} f_n z^{n+1} $$
for any complex number $z$ for which the above series converges. Since this is a particular case of a power series, it is a well known fact (see \cite[Chapter~3]{Rudin}) that there exists a number $0\leq R\leq\infty$, called the \emph{radius of convergence}, such that the series converges absolutely when $|z|<R$ and diverges for $|z|>R$. Observe that nothing is claimed for complex numbers on the circle $|z|=R$. An excellent introductory course to the generating functions with the required prerequisites can be found in \cite[Chapters~10--11]{BenWill} or \cite{GF}.

Take $z\in\mathbb{C}$ such that the series defining $f(z)$ converges absolutely. Multiplying \eqref{eqbasicrecurrence} by $z^{n+1}$ and summing over $n\geq 1$ we get
$$ \sum_{n=1}^{\infty} f_n z^{n+1} = p z \sum_{n=1}^{\infty} f_{n-1} z^n + q \sum_{n=1}^{\infty} \sum_{k=0}^{n-1} f_k z^{k+1} f_{n-k-1} z^{n-k}. $$
Substituting $l=n-k-1$ we see that both $k$ and $l$ range over the nonnegative integers. Rearranging the terms and using $f_0=s$ we obtain
$$ f(z) - sz = p z f(z) + q f(z)^2, $$
which is a quadratic equation for $f(z)$ and can be solved as
\begin{equation}\label{eqgeneratingf}
f(z) = \frac{1-pz \pm \sqrt{(1-pz)^2-4qsz}}{2q}.
\end{equation}
Recall that the square root function above is taken from a complex number and so can be defined as $\sqrt{\zeta}:=e^{(1/2)\mathop{\mathrm{Log}}\zeta}$, where $\mathop{\mathrm{Log}}$ stands for the principal branch of the complex logarithm, i.e. $\mathop{\mathrm{Log}}(\rho e^{i\phi})=\log\rho+i\phi$ for $\rho>0$ and $\phi\in(-\pi,\pi)$. Inspecting the obtained formula we suspect that the ``smallest singularity'' of the generating function is at $z=r$, where $r$ is the smaller of the two positive solutions of the equation $(1-pz)^2=4qsz$, i.e.\@ it is given by Formula \eqref{eqexpressionforr}. Let us remark that we have chosen to write the expression for $r$ in a slightly nonstandard form (with a radical in the denominator), in order to handle both the quadratic case $p\neq 0$ and the linear case $p=0$. The latter case is indeed trivial because $r$ is then the solution of $1=4qsz$.

\begin{lemma}\label{lemma01}
Suppose that the number $r$ is defined by \eqref{eqexpressionforr} and the sequence $(f_n)_{n=0}^{\infty}$ is given by \eqref{eqbasicrecurrence}.
\begin{itemize}
\item[(a)] The series $\sum_{n=1}^{\infty} f_{n-1} r^n$ converges and its sum equals $\frac{1-pr}{2q} = \sqrt{\frac{rs}{q}}$.
\item[(b)] The radius of convergence of the power series $\sum_{n=1}^{\infty} f_{n-1} z^n$ equals $r$ and its sum is given by $f(z)=\frac{1-pz - \sqrt{(1-pz)^2-4qsz}}{2q}$ for each complex number $z$ such that $|z|<r$.
\item[(c)] $f_n \sim \displaystyle\frac{(1+pr)^{1/2}s}{(1-pr)^{3/2}\sqrt{\pi}} \displaystyle\frac{(1/r)^n}{n^{3/2}}$ as $n\to\infty$.\\
    (Here $a_n\sim b_n$ means that the ratio $a_n/b_n$ approaches $1$ in the limit.)
\item[(d)] For any fixed positive integer $k$ the limit $\lim_{n\to\infty}\limits\frac{f_{n-k}}{f_n}$ exists and equals $r^k$.
\end{itemize}
\end{lemma}

Let $w_n$ equal the number of formulas with property $\mathcal{P}$ having exactly $n$ connectives, so that clearly $0\leq w_n\leq f_n$. The two numerical quantities associated with $\mathcal{P}$ that we introduced in Section~\ref{sectionintro} can now be written in terms of the sequence $(w_n)_{n=0}^{\infty}$ as
$$ \mu(\mathcal{P}) = \mu_W = \lim_{n\to\infty}\frac{w_n}{f_n} $$
and
$$ \lambda(\mathcal{P}) = \lambda_W = \sum_{n=0}^{\infty} w_n r^{n+1} = \sum_{n=1}^{\infty} w_{n-1} r^n. $$
By Part (b) of Lemma~\ref{lemma01} we know
$$ \sum_{n=1}^{\infty} w_{n-1} r^n \leq \sum_{n=1}^{\infty} f_{n-1} r^n < \infty, $$
so $\lambda_W$ is actually the value of the generating function of $(w_{n-1})_{n=1}^{\infty}$ at the point $z=r$. In particular,
\begin{equation}\label{eqlambdaf}
\lambda_F = \frac{1-pr}{2q},
\end{equation}
where the letter $F$ represents an arbitrary closed formula, i.e.\@ $\lambda_F$ corresponds to the sequence $(f_n)_{n=0}^{\infty}$. Recall once again that the limit defining $\mu_W$ might not exist. We are indeed primarily interested in $\mu_W$, but the quantity $\lambda_W$ will naturally appear in formulas that follow. In general we will have $0\leq \lambda_W\leq \lambda_F$ and $0\leq \mu_W\leq 1$ whenever the share $\mu_W$ exists.

Equations \eqref{eqformula1a}--\eqref{eqformula2b} in the following lemmas will relate $\lambda$-values and $\mu$-values of sequences connected through quite special recurrence equations. Let us consider three sequences $(u_n)_{n=0}^{\infty}$, $(v_n)_{n=0}^{\infty}$, $(w_n)_{n=0}^{\infty}$ counting certain classes of formulas denoted respectively by letters $U,V,W$. Consequently, $0\leq u_n,v_n,w_n\leq f_n$ for each $n\geq 0$.

We can state the principal result of this section.

\begin{lemma}\label{lemma02}
Take three parameters $\alpha,\beta,\gamma$ satisfying
$$ \alpha,\beta\geq 0, \quad \gamma\in\mathbb{R}, \quad \alpha r+\beta\lambda_F<1, \quad \alpha r+(\beta+2\gamma)\lambda_F<1, $$
where $\lambda_F$ can be substituted from \eqref{eqlambdaf}. Suppose that the sequences $(u_n)_{n=0}^{\infty}$, $(v_n)_{n=0}^{\infty}$, and $(w_n)_{n=0}^{\infty}$ are related by the recurrence relation
\begin{equation}\label{eqrecrelw}
w_0=u_0,\quad w_n = u_n + \alpha w_{n-1} + \sum_{k=0}^{n-1}(\beta v_k+\gamma w_k) \,w_{n-k-1} \ \textrm{ for } n\geq 1.
\end{equation}
Also assume $u_0>0$ and $\beta v_n+\gamma w_n\geq 0$ for each $n$.
\begin{itemize}
\item[(a)] Quantity $\lambda_W$ can be expressed in terms of $\lambda_U$ and $\lambda_V$ as
\begin{equation}\label{eqformula1a}
\lambda_W = \frac{2\lambda_U}{1-\alpha r-\beta\lambda_V +\sqrt{(1-\alpha r-\beta\lambda_V)^2 - 4\gamma\lambda_U}}.
\end{equation}
\item[(b)] If both $\mu_U$ and $\mu_V$ exist, then $\mu_W$ also exists and it is given by
\begin{equation}\label{eqformula1b}
\mu_W = \frac{\mu_U + \beta\lambda_W\mu_V}{1 - \alpha r - \beta\lambda_V - 2\gamma\lambda_W}.
\end{equation}
\end{itemize}
\end{lemma}

Let us remark that $0\leq u_n,v_n,w_n\leq f_n$ actually imposes more restrictive conditions on $\alpha,\beta,\gamma$, but we do not need to perform a detailed discussion of all possible values of the given parameters.

Now we formulate the second auxiliary result.

\begin{lemma}\label{lemma03}
Suppose that the sequence $(w_n)_{n=0}^{\infty}$ is defined using sequences $(u_n)_{n=0}^{\infty}$ and $(v_n)_{n=0}^{\infty}$ as
$$ w_0=0, \quad w_n = \sum_{k=0}^{n-1} u_k v_{n-k-1} \quad\textrm{for } n\geq 1. $$
\begin{itemize}
\item[(a)] Quantity $\lambda_W$ can be expressed in terms of $\lambda_U$ and $\lambda_V$ as
\begin{equation}\label{eqformula2a}
\lambda_W = \lambda_U\lambda_V.
\end{equation}
\item[(b)] If both $\mu_U$ and $\mu_V$ exist, then $\mu_W$ also exists and it is given by
\begin{equation}\label{eqformula2b}
\mu_W = \lambda_V\mu_U+\lambda_U\mu_V.
\end{equation}
\end{itemize}
\end{lemma}

After computing $\mu_W$ for some particular class $W$, one can additionally combine the numerical result with Part (c) of Lemma~\ref{lemma01} to obtain the asymptotics
$$ w_n \sim c_W \frac{(1/r)^n}{n^{3/2}}, \quad \text{as } n\to\infty $$
for some nonnegative constant $c_W$. However, we believe that the relative quantity $\mu_W$ gives a more meaningful measure of size of the class $W$ than the actual numbers $w_n$.

\section{Proofs of the combinatorial results}
\label{sectionproofs}

\begin{proof}[Proof of Lemma~\ref{lemma01}]
(a) Denote $\lambda=\sqrt{rs/q}$, so that the equality $(1-pr)^2=4qrs$ gives $\lambda=(1-pr)/2q$, and it is also straightforward to verify
\begin{equation}\label{eqbasiceqforlambda}
rs + pr\lambda + q\lambda^2 = \lambda.
\end{equation}
Let $S_m=\sum_{n=1}^{m} f_{n-1} r^n$ denote the $m$-th partial sum of the series in question. From \eqref{eqbasicrecurrence} we immediately get
$$ \sum_{n=1}^{m} f_n r^{n+1} \leq p r \sum_{n=1}^{m} f_{n-1} r^n + q \Big(\sum_{j=1}^{m} f_{j-1} z^{j}\Big)^2, $$
so that $S_{m+1} \leq rs + pr S_m + q S_m^2$. Now an easy mathematical induction in $m$ using \eqref{eqbasiceqforlambda} proves the inequality $S_m\leq\lambda$ for each nonnegative integer $m$. Taking the limit as $m\to\infty$ we establish the convergence of $\sum_{n=1}^{\infty} f_{n-1} r^n$. The calculation of the generating function $f(z)$ specialized to $z=r$ gives $f(r)=(1-pr)/2q=\lambda$, which completes the proof of Part (a).

(b) Convergence of the power series for $z=r$ implies that its radius of convergence is at least $r$.
Let us specialize the variable $z$ to the real values from the interval $(r,1/p^2 r)$ if $p\neq0$ and $(r,+\infty)$ if $p=0$. The endpoints of this interval are the solutions of $(1-pz)^2=4qsz$ and the expression \eqref{eqgeneratingf} is not a real number for any such $z$. Therefore the series does not converge on this whole interval, which implies that its radius of convergence is at most $r$. We still need to decide about the sign before the square root in \eqref{eqgeneratingf}. Since $f$ is defined as a sum of the power series $\sum_{n=1}^{\infty} f_{n-1} z^n$, it constitutes a complex-analytic function on the open disk around $0$ of radius $r$ and it is also clear that $f(0)=0$. The quadratic polynomial $(1-pz)^2-4qsz$ does not have any zeroes on that same disk, so simply by continuity of $f$ we must keep the same sign in \eqref{eqgeneratingf} on the whole disk. Plugging in $z=0$ we see that only the negative sign is possible. This proves Part (b) of the lemma.

(c) For this part we need to use a more advanced result from complex analysis. Using the Vi\`{e}te formulas for the quadratic polynomial with roots $r$ and $\frac{1}{p^2 r}$ we can write
$$ (1-pz)^2-4qsz = p^2 (z-r) \Big(z-\frac{1}{p^2 r}\Big) = \Big(1-\frac{z}{r}\Big)(1-p^2 r z), $$
so the expression for the generating function can be rewritten as
\begin{equation}\label{eqfgenas}
f(z) = \frac{-1}{2q}(1-p^2 r z)^{1/2} \Big(1-\frac{z}{r}\Big)^{1/2} + \frac{1-pz}{2q} .
\end{equation}
Now we use the following result.

\begin{theorem}[from \cite{FO}]\label{theoremaux}
Suppose that the generating function $a$ of a sequence $(a_n)_{n=0}^{\infty}$ has $r>0$ as the radius of convergence and that it is complex-analytic in the domain
$$ \Omega:=\big\{z\in\mathbb{C}\setminus\{r\} \,:\, |z|<r+\epsilon,\ |\mathrm{Arg}(z-r)|>\psi\big\} $$
for some $\epsilon>0$ and $0<\psi<\pi/2$. Here $\mathrm{Arg}$ denotes the principal argument function defined by $\mathrm{Arg}(\rho e^{i\phi}):=\phi$ for $\rho>0$ and $\phi\in(-\pi,\pi]$. Also suppose that
$$ a(z) = b(z)\Big(1-\frac{z}{r}\Big)^{\theta} + c(z) $$
for some $\theta\in\mathbb{R}\setminus\{0,1,2,\ldots\}$ and for some complex-analytic functions $b$ and $c$ on a neighborhood of $z=r$. If
$$ K := \lim_{\Omega\ni z\to r} b(z)\neq 0, $$
then
$$ a_n \sim \frac{K(1/r)^n}{\mathrm{\Gamma}(-\theta)n^{\theta+1}} \quad\text{as } n\to\infty. $$
\end{theorem}

This result is merely a reformulation of \cite[Corollary~2]{FO}: One handles the term $b(z)(1-z/r)^{\theta}$ by applying Part (i) with $\alpha=\theta$, while the term $c(z)$ is handled using Part (ii) with $\alpha=0$. Actually, the results in \cite{FO} are formulated for $r=1$ only, which is not a loss of generality by rescaling.

Applying Theorem~\ref{theoremaux} to \eqref{eqfgenas} with $\theta=1/2$ we get
$$ f_{n-1} \sim \frac{-(1/2q)(1-p^2 r^2)^{1/2}(1/r)^n}{\mathrm{\Gamma}(-1/2)\,n^{3/2}} \quad\text{as } n\to\infty. $$
The last expression easily transforms into the one stated in the lemma by replacing $n$ with $n+1$ and using
\,$\mathrm{\Gamma}(-1/2)=-2\sqrt{\pi}$, \,$q=(1-pr)^2/4rs$, \,$(n+1)^{3/2}\sim n^{3/2}$.

(d) This part is an immediate consequence of Part (c) and $(n-k)^{3/2}\sim n^{3/2}$ as $n\to\infty$:
$$ \frac{f_{n-k}}{f_n} \sim \frac{(1/r)^{n-k}(n-k)^{-3/2}}{(1/r)^n n^{-3/2}} \sim r^k. $$
\end{proof}

\smallskip
The generating functions of the sequences in question will be denoted using the same lowercase letters. From Lemma \ref{lemma01}(a) we know that they converge absolutely on the interval $[0,r]$.

\begin{proof}[Proof of Lemma~\ref{lemma02}]
(a) This part is once again quite standard, but we include it for completeness. Suppose $\gamma\neq 0$ as the other case is straightforward. Let us multiply the recurrence relation in \eqref{eqrecrelw} by $z^{n+1}$ and sum over $n\geq 1$. That way we obtain
\begin{align*}
\sum_{n=1}^{\infty} w_n z^{n+1} & = \sum_{n=1}^{\infty} u_n z^{n+1} + \alpha z \sum_{n=1}^{\infty} w_{n-1} z^n \\
& \quad + \sum_{n=1}^{\infty}\sum_{k=0}^{n-1} (\beta v_k z^{k+1} +\gamma w_k z^{k+1}) \,w_{n-k-1}z^{n-k}.
\end{align*}
Adding $u_0 z$ to both sides, this equality can be rewritten as
\begin{equation}\label{eqequforw}
w(z) = u(z) + \alpha z w(z) + \big(\beta v(z)+\gamma w(z)\big)w(z).
\end{equation}
Solving the quadratic equation for $w(z)$ and specializing to the real values of the variable $z$, more precisely $z=t\in[0,r]$, gives
\begin{equation}\label{eqgenfun}
w(t) = \frac{1-\alpha t-\beta v(t)\pm\sqrt{(1-\alpha t-\beta v(t))^2-4\gamma u(t)}}{2\gamma}.
\end{equation}
Similarly as in the previous lemma we have to decide about the sign before the square root in the above formula for $w$, but the argument will be slightly more complicated, as the function under the square root sign need not be a polynomial anymore.

Recall that we have
$$ 0\leq u(t) = \sum_{n=0}^{\infty} u_n t^{n+1} \leq \sum_{n=0}^{\infty} f_n r^{n+1} = \lambda_F \ \textrm{ for }t\in[0,r] $$
and
$$ u'(t) = \sum_{n=0}^{\infty} n u_n t^n >0 \ \textrm{ for }t\in(0,r), $$
since the power series can be differentiated term-by-term on its disk of convergence; see \cite[Chapter~8]{Rudin}. The same is true for $v$ and $w$. If $\gamma<0$, then only the negative sign in \eqref{eqgenfun} will yield nonnegative values of $w(t)$, so we do not really have a choice. Therefore we assume $\gamma>0$. Taking $t=0$ and using $u(0)=v(0)=w(0)=0$, we see that only the minus sign in \eqref{eqgenfun} is possible in a neighborhood of $t=0$. We claim that the same formula holds for all $0\leq t\leq r$. By continuity of $w$, the sign before the square root could only change at a point $t_0\in(0,r)$ such that $\varphi(t_0)=0$, where
$$ \varphi(t) := \big(1-\alpha t-\beta v(t)\big)^2-4\gamma u(t). $$
Since the quadratic equation for $w(t)$ must have a real solution, we know that its discriminant is nonnegative, i.e.\@ $\varphi(t)\geq 0$ for $0\leq t\leq r$. However, $\varphi$ extends to a complex-analytic function on a disk around $0$ with radius $r$ and $t_0$ is its root of even order, which implies $\varphi'(t_0)=0$. On the other hand, from the expression for $\varphi$ we obtain
$$ \varphi'(t) = -2\big(\underbrace{1-\alpha t-\beta v(t)}_{>0}\big)\big(\underbrace{\alpha+\beta v'(t)}_{\geq 0}\big)-4\gamma\underbrace{u'(t)}_{>0} < 0 $$
for each $0<t<r$, which is a contradiction. Note that we have used
$$ \alpha t + \beta v(t) < \alpha r + \beta \lambda_F < 1 . $$
We conclude that only the minus sign is possible in \eqref{eqgenfun}. Equation \eqref{eqformula1a} is obtained by letting $t=r$ and rewriting \eqref{eqgenfun} so that the square root ends up in the denominator. It is trivial to verify the same formula when $\gamma=0$, by solving the linear equation \eqref{eqequforw} for $w(z)$ and plugging in $z=r$.

(b) The main difficulty is that we do not yet know the existence of the limit $\lim_{n\to\infty}\limits w_n/f_n$. However the limit inferior and the limit superior of a bounded sequence can be defined as the smallest and the largest cluster point of the sequence (see \cite[Chapter~3]{Rudin}) and they are always well-defined real numbers. For that reason we denote
\begin{equation}\label{eqliminflimsup}
\underline{\mu}_W:=\liminf_{n\to\infty}\frac{w_n}{f_n}, \quad \overline{\mu}_W:=\limsup_{n\to\infty}\frac{w_n}{f_n}.
\end{equation}
From $0\leq w_n\leq f_n$ we know that $0\leq\underline{\mu}_W\leq\overline{\mu}_W\leq 1$. Our goal is to show $\underline{\mu}_W=\overline{\mu}_W$ and to evaluate this common value.

Let us take positive integers $m,n$ such that $n>2m$ and split the sum in the defining recurrence relation \eqref{eqrecrelw} into three parts:
\begin{align*}
w_n & = u_n + \alpha w_{n-1} + \underbrace{\sum_{k=0}^{m-1} w_{n-k-1} (\beta v_k + \gamma w_k)}_{a_{m,n}}
+ \underbrace{\sum_{k=0}^{m-1} (\beta v_{n-k-1} + \gamma w_{n-k-1}) w_k}_{b_{m,n}} \\[-2mm]
& \quad + \underbrace{\sum_{k=0}^{n-2m-1} (\beta v_{m+k} + \gamma w_{m+k}) w_{n-m-k-1}}_{c_{m,n}}.
\end{align*}
Informally, we extract $m$ summands from each side of the summation range to constitute ``dominant parts'' $a_{m,n}$ and $b_{m,n}$, while the remainder $c_{m,n}$ will be estimated as comparably small.

Dividing by $f_n$ we obtain the splitting
\begin{equation}\label{eqsplitting}
\frac{w_n}{f_n} = \frac{u_n}{f_n} + \frac{\alpha w_{n-1}}{f_n} + \frac{a_{m,n}}{f_n} + \frac{b_{m,n}}{f_n} + \frac{c_{m,n}}{f_n}.
\end{equation}
Using Part (d) of Lemma~\ref{lemma01} and basic properties of $\liminf$ and $\limsup$ we obtain
{\allowdisplaybreaks\begin{align*}
& \liminf_{n\to\infty}\frac{w_{n-k-1}}{f_n} = \liminf_{n\to\infty}\Big(\frac{f_{n-k-1}}{f_n}\frac{w_{n-k-1}}{f_{n-k-1}}\Big) \\
& \qquad\qquad\qquad\ \  = \lim_{n\to\infty}\!\frac{f_{n-k-1}}{f_n} \,\liminf_{n\to\infty}\frac{w_{n-k-1}}{f_{n-k-1}} = r^{k+1} \underline{\mu}_W, \\
& \limsup_{n\to\infty}\frac{w_{n-k-1}}{f_n} = r^{k+1} \overline{\mu}_W, \\
& \liminf_{n\to\infty}\frac{\beta v_{n-k-1}+\gamma w_{n-k-1}}{f_n} =
\left\{\begin{array}{cl} r^{k+1} (\beta\mu_V + \gamma\overline{\mu}_W) & \textrm{ if } \gamma<0, \\
r^{k+1} (\beta\mu_V + \gamma\underline{\mu}_W) & \textrm{ if } \gamma\geq 0, \end{array}\right. \\
& \limsup_{n\to\infty}\frac{\beta v_{n-k-1}+\gamma w_{n-k-1}}{f_n} =
\left\{\begin{array}{cl} r^{k+1} (\beta\mu_V + \gamma\underline{\mu}_W) & \textrm{ if } \gamma<0, \\
r^{k+1} (\beta\mu_V + \gamma\overline{\mu}_W) & \textrm{ if } \gamma\geq 0. \end{array}\right.
\end{align*}}
Consequently,
\begin{equation}\label{eqextraest}
\alpha r\underline{\mu}_W \leq\liminf_{n\to\infty}\frac{\alpha w_{n-1}}{f_n} \leq \limsup_{n\to\infty}\frac{\alpha w_{n-1}}{f_n} \leq \alpha r\overline{\mu}_W
\end{equation}
and
\begin{align}
& \underline{\mu}_W\sum_{k=0}^{m-1} (\beta v_k + \gamma w_k)r^{k+1}
\leq\liminf_{n\to\infty}\frac{a_{m,n}}{f_n} \nonumber \\[-1mm]
& \leq \limsup_{n\to\infty}\frac{a_{m,n}}{f_n}
\leq \overline{\mu}_W\sum_{k=0}^{m-1} (\beta v_k + \gamma w_k)r^{k+1}. \label{eqpest}
\end{align}
Similarly, for $\gamma<0$ we have
\begin{align}
& (\beta\mu_V + \gamma\overline{\mu}_W)\sum_{k=0}^{m-1} w_k r^{k+1}
\leq\liminf_{n\to\infty}\frac{b_{m,n}}{f_n} \nonumber \\[-1mm]
& \leq \limsup_{n\to\infty}\frac{b_{m,n}}{f_n}
\leq (\beta\mu_V + \gamma\underline{\mu}_W)\sum_{k=0}^{m-1} w_k r^{k+1}, \label{eqqest1}
\end{align}
while for $\gamma>0$
\begin{align}
& (\beta\mu_V + \gamma\underline{\mu}_W)\sum_{k=0}^{m-1} w_k r^{k+1}
\leq\liminf_{n\to\infty}\frac{b_{m,n}}{f_n} \nonumber \\[-1mm]
& \leq \limsup_{n\to\infty}\frac{b_{m,n}}{f_n}
\leq (\beta\mu_V + \gamma\overline{\mu}_W)\sum_{k=0}^{m-1} w_k r^{k+1}. \label{eqqest2}
\end{align}
It remains to estimate $\frac{c_{m,n}}{f_n}$ using the recurrence relation \eqref{eqbasicrecurrence} as
{\allowdisplaybreaks\begin{align*}
0\leq \frac{c_{m,n}}{f_n} & \leq \frac{\beta+|\gamma|}{f_n}\sum_{k=0}^{n-2m-1}f_{m+k}f_{n-m-k-1} \\
& = \frac{\beta+|\gamma|}{q f_n}\bigg(f_n - p f_{n-1} - 2q\sum_{k=0}^{m-1}f_{k}f_{n-k-1}\bigg) \\
& = \frac{\beta+|\gamma|}{q}\bigg(1 - \frac{p f_{n-1}}{f_n} - 2q\sum_{k=0}^{m-1} \frac{f_{n-k-1}}{f_n}f_k\bigg),
\end{align*}}
so from Lemma \ref{lemma01}(d)
\begin{equation}\label{eqremainder}
\limsup_{n\to\infty}\frac{c_{m,n}}{f_n} \leq \frac{\beta+|\gamma|}{q}\bigg(1 - pr - 2q\sum_{k=0}^{m-1} f_k r^{k+1}\bigg).
\end{equation}

Now we distinguish two cases.

\emph{Case $\gamma<0$.} Combining \eqref{eqsplitting}, \eqref{eqextraest}, \eqref{eqpest}, \eqref{eqqest1}, and \eqref{eqremainder} gives us
\begin{align*}
\overline{\mu}_W & \leq \mu_U + \limsup_{n\to\infty}\frac{\alpha w_{n-1}}{f_n} + \limsup_{n\to\infty}\frac{a_{m,n}}{f_n} + \limsup_{n\to\infty}\frac{b_{m,n}}{f_n} + \limsup_{n\to\infty}\frac{c_{m,n}}{f_n} \\
& \leq \mu_U + \alpha r\overline{\mu}_W + \overline{\mu}_W\sum_{k=0}^{m-1} (\beta v_k + \gamma w_k)r^{k+1}
+  (\beta\mu_V + \gamma\underline{\mu}_W)\sum_{k=0}^{m-1} w_k r^{k+1} \\
& \quad + \frac{\beta+|\gamma|}{q}\bigg(1 - pr - 2q\sum_{k=0}^{m-1} f_k r^{k+1}\bigg)
\end{align*}
and similarly
$$ \underline{\mu}_W \geq \mu_U + \alpha r\underline{\mu}_W + \underline{\mu}_W\sum_{k=0}^{m-1} (\beta v_k + \gamma w_k)r^{k+1}
+ (\beta\mu_V + \gamma\overline{\mu}_W)\sum_{k=0}^{m-1} w_k r^{k+1}. $$
From \eqref{eqlambdaf} we get
$$ 1 - pr - 2q\sum_{k=0}^{\infty} f_k r^{k+1} = 1 - pr - 2q\lambda_F = 0, $$
so letting $m\to\infty$ we end up with a system of inequalities:
\begin{align*}
\overline{\mu}_W & \leq \mu_U + \alpha r\overline{\mu}_W + \big(\beta \lambda_V+\gamma \lambda_W\big)\overline{\mu}_W + \lambda_W(\beta\mu_V + \gamma\underline{\mu}_W), \\
\underline{\mu}_W & \geq \mu_U + \alpha r\underline{\mu}_W + \big(\beta \lambda_V+\gamma \lambda_W\big)\underline{\mu}_W + \lambda_W(\beta\mu_V + \gamma\overline{\mu}_W),
\end{align*}
which can in turn be rewritten as
$$ \big(1-\alpha r-\beta \lambda_V\big)\overline{\mu}_W \leq \mu_U + \lambda_W \big(\beta\mu_V + \gamma \underline{\mu}_W + \gamma \overline{\mu}_W\big) \leq \big(1-\alpha r-\beta \lambda_V\big)\underline{\mu}_W. $$
One of the conditions on the coefficients $\alpha,\beta$ gives
$$ 1-\alpha r-\beta \lambda_V \geq 1-\alpha r-\beta \lambda_F > 0, $$
so dividing by this number we conclude $\overline{\mu}_W\leq\underline{\mu}_W$, i.e.\@ $\overline{\mu}_W=\underline{\mu}_W$, and also that the above inequalities become equalities. Therefore the share $\mu_W$ exists and it satisfies the equation
$$ \mu_W = \mu_U + \alpha r\mu_W + \big(\beta \lambda_V+\gamma \lambda_W\big)\mu_W + \lambda_W(\beta\mu_V + \gamma\mu_W), $$
which easily transforms into \eqref{eqformula1b}.

\emph{Case $\gamma>0$.}
In exactly the same way as in the previous case inequalities \eqref{eqextraest}, \eqref{eqpest}, \eqref{eqqest2}, and \eqref{eqremainder} imply
\begin{align*}
\overline{\mu}_W & \leq \mu_U + \alpha r\overline{\mu}_W + \big(\beta \lambda_V+\gamma \lambda_W\big)\overline{\mu}_W + \lambda_W(\beta\mu_V + \gamma\overline{\mu}_W), \\
\underline{\mu}_W & \geq \mu_U + \alpha r\underline{\mu}_W + \big(\beta \lambda_V+\gamma \lambda_W\big)\underline{\mu}_W + \lambda_W(\beta\mu_V + \gamma\underline{\mu}_W).
\end{align*}
If we rewrite these estimates as
$$ \big(1-\alpha r-\beta \lambda_V-2\gamma\lambda_W\big)\overline{\mu}_W \leq \mu_U + \beta \lambda_W\mu_V \leq \big(1-\alpha r-\beta \lambda_V-2\gamma\lambda_W\big)\underline{\mu}_W, $$
then with an aid of
$$ 1-\alpha r-\beta \lambda_V-2\gamma\lambda_W \geq 1-\alpha r-(\beta+2\gamma)\lambda_F >0 $$
we arrive at $\overline{\mu}_W=\underline{\mu}_W$ and \eqref{eqformula1b} once again.
\end{proof}

\smallskip
The proof of Lemma~\ref{lemma03} is much shorter, but it follows the same basic idea.

\begin{proof}[Proof of Lemma~\ref{lemma03}]
(a) This part is a direct consequence of
$$ w(z)=\sum_{n=1}^{\infty}\sum_{k=0}^{n-1} u_k z^{k+1} v_{n-k-1} z^{n-k} = \Big(\sum_{k=0}^{\infty}u_k z^{k+1}\Big) \Big(\sum_{l=0}^{\infty}v_l z^{l+1}\Big) = u(z) v(z). $$

(b) Imitating the proof of Lemma~\ref{lemma02} we split
$$ \sum_{k=0}^{n-1} u_k v_{n-k-1} = \underbrace{\sum_{k=0}^{m-1} u_{n-k-1} v_k}_{a_{m,n}} + \underbrace{\sum_{k=0}^{m-1} v_{n-k-1} u_k}_{b_{m,n}} + \underbrace{\sum_{k=0}^{n-2m-1} u_{m+k} v_{n-m-k-1}}_{c_{m,n}}. $$
Once again we have denoted dominant parts of the sum by $a_{m,n}$ and $b_{m,n}$ and the remainder by $c_{m,n}$.
Exactly as before we easily derive:
{\allowdisplaybreaks\begin{align*}
& \lim_{n\to\infty}\frac{u_{n-k-1}}{f_n} = \mu_U r^{k+1},\quad \lim_{n\to\infty}\frac{v_{n-k-1}}{f_n} = \mu_V r^{k+1}, \\
& \lim_{n\to\infty}\frac{a_{m,n}}{f_n} = \mu_U\sum_{k=0}^{m-1} v_k r^{k+1},\quad \lim_{n\to\infty}\frac{b_{m,n}}{f_n} = \mu_V\sum_{k=0}^{m-1} u_k r^{k+1}, \\
& \limsup_{n\to\infty}\frac{c_{m,n}}{f_n} \leq \frac{1}{q}\bigg(1 - pr - 2q\sum_{k=0}^{m-1} f_k r^{k+1}\bigg),
\end{align*}}
which in turn implies
\begin{align*}
\limsup_{n\to\infty}\frac{w_n}{f_n} & \leq \mu_U\!\sum_{k=0}^{m-1}\! v_k r^{k+1} + \mu_V\!\sum_{k=0}^{m-1}\! u_k r^{k+1} + \frac{1}{q}\bigg(1 - pr - 2q\!\sum_{k=0}^{m-1}\! f_k r^{k+1}\bigg), \\
\liminf_{n\to\infty}\frac{w_n}{f_n} & \geq \mu_U\sum_{k=0}^{m-1} v_k r^{k+1} + \mu_V\sum_{k=0}^{m-1} u_k r^{k+1}.
\end{align*}
By letting $m\to\infty$ we obtain that the common value of the lower limit and the upper limit of $(w_n/f_n)_{n=0}^{\infty}$ equals
\begin{equation*}\mu_U\sum_{k=0}^{\infty} v_k r^{k+1} + \mu_V\sum_{k=0}^{\infty} u_k r^{k+1} = \mu_U\lambda_V + \mu_V\lambda_U.
\end{equation*}
\end{proof}

\smallskip
Finally, we can return to our rather general result on classes of formulas generated by certain production rules.

\begin{proof}[Proof of Theorem~\ref{theorem01}]
Throughout the proof we will denote $t_n=\#\mathcal{T}_n$, $u_n=\#\mathcal{U}_n$, $v_n=\#\mathcal{V}_n$, $w_n=\#\mathcal{W}_n$.

(a) This part is an immediate consequence of $w_n=u_n+v_n$.

(b) The production rules lead to the recurrence relation
$$ w_n = u_n + p'w_{n-1} + q'\sum_{k=0}^{n-1}w_k w_{n-k-1} + q''\sum_{k=0}^{n-1}(v_k-w_k)w_{n-k-1} $$
with the initial condition $w_0 = u_0 \geq 1$. We want to apply Lemma~\ref{lemma02} with $\alpha=p'$, $\beta=q''$, and $\gamma=q'-q''$. Since the condition $q''v_n+(q'-q'')w_n\geq 0$ is obvious, we only need to verify the requirements on the coefficients $\alpha,\beta,\gamma$. We have
$$ \alpha r + \beta\lambda_F = p'r + q''\frac{1-pr}{2q} \leq pr + \frac{1-pr}{2} < 1 $$
and
$$ \alpha r + (\beta+2\gamma)\lambda_F = p'r + (2q'-q'')\frac{1-pr}{2q} \leq pr + 2q\frac{1-pr}{2q} = 1 . $$
The last inequality is strict unless $(p',q',q'')=(p,q,0)$, which was not allowed in the statement of Theorem~\ref{theorem01}. Therefore, Lemma~\ref{lemma02} applies and we simply need to use Formulas~\eqref{eqformula1a} and \eqref{eqformula1b}.

(c) This time the productions give
$$ w_0=s',\quad w_n = p't_{n-1} + q'\sum_{k=0}^{n-1}u_k v_{n-k-1}\textrm{ for }n\geq 1 $$
and we are going to use Lemma~\ref{lemma03}. Writing
$$ \sum_{n=0}^{\infty} w_n r^{n+1} = s'r + p'r\sum_{n=1}^{\infty} t_{n-1}r^n + q'\sum_{n=1}^{\infty}\bigg(\sum_{k=0}^{n-1}u_k v_{n-k-1}\bigg)r^{n+1} $$
and applying \eqref{eqformula2a} to the last term we obtain
$$ \lambda_W = s'r + p'r\lambda_T + q'\lambda_U\lambda_V. $$
Also, applying Lemma~\ref{lemma01}(d) to the first and Formula~\eqref{eqformula2b} to the second term in
$$ \frac{w_n}{f_n} = p'\frac{f_{n-1}}{f_n}\frac{t_{n-1}}{f_{n-1}} + q'\frac{1}{f_n}\sum_{k=0}^{n-1}u_k v_{n-k-1} $$
we conclude that $\mu_W$ exists and
$$ \mu_W = p'r\mu_T + q'(\lambda_V\mu_U+\lambda_U\mu_V). $$

(d) From the production rules we know $w_0=u_0$, $w_1=u_1+p'w_0$, and
$$ w_n = u_n + p'w_{n-1} + q'w_{n-2}\textrm{ for } n\geq 2. $$
Summing in $n$ we obtain
$$ \sum_{n=2}^{\infty}w_n r^{n+1} = \sum_{n=2}^{\infty}u_n r^{n+1} + p'r\sum_{n=2}^{\infty}w_{n-1} r^n + q'r^2\sum_{n=2}^{\infty}w_{n-2}r^{n-1}, $$
i.e.
$$ \lambda_W - u_0 r - (u_1+p'u_0)r^2 = \lambda_U - u_0 r - u_1 r^2 + p'r(\lambda_W - u_0 r) + q'r^2\lambda_W, $$
which in turn simplifies as
$$ (1-p'r-q'r^2) \lambda_W = \lambda_U. $$
We are allowed to divide by $1-p'r-q'r^2$ because
$$ p'r + q'r^2 < (p+q)r < (p+q)\frac{1}{p+2q} < 1. $$
Finally, utilizing
$$ \frac{w_n}{f_n} = \frac{u_n}{f_n} + p'\frac{f_{n-1}}{f_n}\frac{w_{n-1}}{f_{n-1}} + q'\frac{f_{n-2}}{f_n}\frac{w_{n-2}}{f_{n-2}} $$
and Lemma~\ref{lemma01}(d) we derive
\begin{align*}
\overline{\mu}_W & \leq \mu_U + p'r\overline{\mu}_W + q'r^2\overline{\mu}_W, \\
\underline{\mu}_W & \geq \mu_U + p'r\underline{\mu}_W + q'r^2\underline{\mu}_W,
\end{align*}
where we use the notation \eqref{eqliminflimsup} once again. Analogously as in the proof of Lemma~\ref{lemma02} we conclude
$$ \underline{\mu}_W = \overline{\mu}_W = \frac{\mu_U}{1-p'r-q'r^2}. $$
\end{proof}

For the purposes of applications it will sometimes be useful to interpret Theorem~\ref{theorem01}(a) as:
$$ \lambda_V = \lambda_W - \lambda_U,\quad \mu_V = \mu_W - \mu_U. $$

\section{Calculating the shares}
\label{sectioncalculations}

\begin{proof}[Proof of Theorem~\ref{theorem02}]
(a) In this case $p=0$, $q=2$, $s=1$, and consequently $r=1/8$, $\lambda_F=1/4$, $\mu_F=1$ by Formulas~\eqref{eqexpressionforr} and \eqref{eqlambdaf}.

\emph{Negative, direct, and cross formulas}.
We begin by discussing negative formulas $N$, which are given by
$$ N \,::=\, \bot \ | \ (F\setminus N) \rightarrow N, $$
so we can apply Theorem~\ref{theorem01}(b) with
$$ p'=0,\ q'=0,\ q''=1,\ U=\{\bot\},\ V=F. $$
Since $\lambda_U=r=1/8$, $\mu_U=0$, $\lambda_V=\lambda_F=1/4$, $\mu_V=\mu_F=1$, we obtain
\begin{align*}
\lambda_N & = \frac{2r}{1-\lambda_F + \sqrt{(1-\lambda_F)^2+4r}} = \frac{1}{3+\sqrt{17}} = 0.14038\ldots, \\
\mu_N & = \frac{\lambda_N \mu_F}{1 - \lambda_F + 2\lambda_N} = \frac{1}{2}-\frac{3}{2\sqrt{17}} = 0.13619\ldots.
\end{align*}
Therefore, approximately $14\%$ of all \ILC\ formulas are negative. It is interesting to remark that the analogues of negative formulas when $p=0$, $q=1$, $s=1$ constitute the sequence A055113 in \cite{OEIS} and are ``counted'' in Examples 11.3 and 11.31 of the textbook \cite{BenWill}. The reader can compare our technique with explicit manipulations with a quite complicated generating function presented in \cite{BenWill}.

Direct formulas $M$ are really similar to negative ones, in the sense that $\mathcal{M}_0=\emptyset$ and $\mathcal{M}_n=\mathcal{N}_n$ for $n\geq 1$. Therefore,
\begin{align*}
\lambda_M & = \lambda_N - r = \frac{8\lambda_N-1}{8} = 0.01538\ldots, \\
\mu_M & = \mu_N = 0.13619\ldots.
\end{align*}

Finally, cross formulas $X$ are of the form $M\rightarrow\bot$ for some direct formula $M$ and counting them is the same as counting $\diamond M$ for some unary operator $\diamond$. Thus, Theorem~\ref{theorem01}(c) applies with
$$ p'=1,\ q'=0,\ s'=0,\ T=M,\ U=V=\emptyset $$
and gives
\begin{align*}
\lambda_X & = r\lambda_M = \frac{8\lambda_N-1}{64} = 0.00192\ldots, \\
\mu_X & = r\mu_M = \frac{\mu_N}{8} = 0.01702\ldots.
\end{align*}

\emph{Basic locally \GL-like formulas}.
Let us compute shares of each of the five classes of formulas mentioned in Theorem~\ref{theoremmc}.
\begin{itemize}
\item
Formulas of the form $\bot\rhd F$. We apply Theorem~\ref{theorem01}(c) with
$$ p'=1,\ q'=0,\ s'=0,\ T=F,\ U=V=\emptyset $$
to get
$$ \lambda(\bot\rhd F) = r\lambda_F = \frac{1}{32} \quad\textrm{and}\quad \mu(\bot\rhd F) = r\mu_F = \frac{1}{8}. $$

\item
Formulas of the form $F\rhd A$. This time we can make use of Theorem~\ref{theorem01}(c) by choosing
$$ p'=0,\ q'=1,\ s'=0,\ T=\emptyset,\ U=F,\ V=A=F\setminus N, $$
so
\begin{align*}
\lambda(F\rhd A) & = \lambda_F \big(\lambda_F-\lambda_N\big) = \frac{1-4\lambda_N}{16}, \\
\mu(F\rhd A) & = (\lambda_F-\lambda_N)\mu_F+\lambda_F(\mu_F-\mu_N) = \frac{2-4\lambda_N-\mu_N}{4}.
\end{align*}

\item
Formulas of the form $A\rhd\bot$. Theorem~\ref{theorem01}(c) with
$$ p'=1,\ q'=0,\ s'=0,\ T=A=F\setminus N,\ U=V=\emptyset $$
gives
$$ \lambda(A\rhd\bot) = r(\lambda_F-\lambda_N) = \frac{1-4\lambda_N}{32},\quad \mu(A\rhd\bot) = r(\mu_F-\mu_N) = \frac{1-\mu_N}{8}. $$

\item
Formulas of the form $N\rhd X$. Theorem~\ref{theorem01}(c) applies again, this time with
$$ p'=0,\ q'=1,\ s'=0,\ T=\emptyset,\ U=N,\ V=X, $$
so
\begin{align*}
\lambda(N\rhd X) & = \lambda_N \lambda_X = \frac{\lambda_N (8\lambda_N -1)}{64}, \\
\mu(N\rhd X) & = \lambda_X \mu_N + \lambda_N \mu_X = \frac{(16\lambda_N -1)\mu_N}{64}.
\end{align*}

\item
Formulas of the form $X\rhd\bot$. Finally, we take
$$ p'=1,\ q'=0,\ s'=0,\ T=X,\ U=V=\emptyset $$
in Theorem~\ref{theorem01}(c):
$$ \lambda(X\rhd\bot) = r\lambda_X = \frac{8\lambda_N -1}{8^3} \quad\textrm{and}\quad \mu(X\rhd\bot) = r\mu_X = \frac{\mu_N}{64}. $$
\end{itemize}

Unfortunately, counting basic locally \GL-like formulas is not just summing the cardinalities of classes from Theorem~\ref{theoremmc}, because those classes are not disjoint. However, they are almost disjoint: The only pairs having nonempty intersection are $\bot\rhd F$ and $F\rhd A$, and $\bot\rhd F$ and $N\rhd X$. There are no triple intersections. (It helps to first observe that all direct formulas are affirmative and that all cross formulas are negative.) Moreover, counting those intersections is easy, and we have already done all the necessary calculations.
\begin{itemize}
\item Intersecting $\bot\rhd F$ and $F\rhd A$ we get $\bot\rhd A$. Counting those is the same as counting $A\rhd\bot$, which we have already done:
$$ \lambda(\bot\rhd F\ \textrm{and}\ F\rhd A)=\lambda(A\rhd\bot), \quad \mu(\bot\rhd F\ \textrm{and}\ F\rhd A)=\mu(A\rhd\bot). $$

\item Intersecting $\bot\rhd F$ and $N\rhd X$ we get $\bot\rhd X$. Counting those is the same as counting $X\rhd\bot$, so
$$ \lambda(\bot\rhd F\ \textrm{and}\ N\rhd X)=\lambda(X\rhd\bot), \quad \mu(\bot\rhd F\ \textrm{and}\ N\rhd X)=\mu(X\rhd\bot). $$
\end{itemize}

Now we can enumerate the basic locally \GL-like formulas by subtracting the formulas we counted twice (i.e.\@ using the so-called \emph{inclusion-exclusion principle}), and in doing so we can even cancel out some terms. This finally leads us to the following quantities:
\begin{align*}
\lambda_B & = \param{basic\ locally\ \GL-like} = \lambda(\bot\rhd F) + \lambda(F\rhd A) + \lambda(N\rhd X) \\
& \,= \frac{8\lambda_N^2 - 17\lambda_N + 6}{64} = 0.05892\ldots, \\
\mu_B & = \share{basic\ locally\ \GL-like} = \mu(\bot\rhd F) + \mu(F\rhd A) + \mu(N\rhd X) \\
& \,= \frac{16\lambda_N \mu_N - 64\lambda_N - 17\mu_N + 40}{64} = 0.45321\ldots.
\end{align*}

\emph{Locally \GL-like formulas}.
Now we tackle all locally \GL-like formulas $L$, which are given by the productions
$$ L \,::=\, \bot \ | \ B \ | \ L \rightarrow L. $$
Once again, Theorem~\ref{theorem01}(b) is tailored for this situation and we need to choose
$$ p'=0,\ q'=1,\ q''=0,\ U=\{\bot\}\cup B,\ V=\emptyset $$
and observe $\lambda_U=\lambda_B+r$, $\mu_U=\mu_B$ in order to evaluate
\begin{align*}
\lambda_L & = \frac{2(\lambda_B+r)}{1+\sqrt{1-4(\lambda_B+r)}} = 0.24294\ldots, \\
\mu_L & = \frac{\mu_B}{1-2\lambda_L} = 0.88155\ldots.
\end{align*}

\emph{\GL-like formulas}.
So far we know that more than $88\%$ of \ILC\ formulas, in the sense we described, have normal forms from \GL. If we want to calculate the share of general \GL-like formulas, as they are defined in Section~\ref{sectiondefinitions}, we have to consider a few more cases and use the principle of inclusion and exclusion once again. Fortunately, many of the intersecting cases are either empty or already covered. For example, there are no formulas that are both negative and direct, since the main connective is different. For the same reason no formula can be both negative and of the form $(G\rightarrow\bot)\rhd\bot$.

As the first instance of an already covered case we will investigate locally \GL-like formulas that are also of the form $(G\rightarrow\bot)\rhd\bot$ for some smaller \GL-like formula $G$. First, we realize that we only need to consider basic locally \GL-like formulas, since other locally \GL-like formulas do not  have $\rhd$ as the main connective. Then, of those, we only need to consider $A\rhd\bot$ and $X\rhd\bot$ forms, since other forms either have $\bot$ on the left, which is not a conditional, or have $A$ or $X$ on the right, which cannot be $\bot$.

Thus, we have reduced the problem to finding formulas that are either affirmative or cross, and at the same time of the form $G\rightarrow\bot$, and we append ``$\rhd\bot$'' at the end. Cross formulas are already of the form $M\rightarrow\bot$, and we know that direct formulas are \GL-like, so all of them are counted. Looking at the possible forms for affirmative formulas, we find that they can only be of the form $N\rightarrow\bot$. Summing it up, we see that the desired formulas are either of the form $(N\rightarrow\bot)\rhd\bot$, or of the form $(M\rightarrow\bot)\rhd\bot$. The number of such formulas with complexity $n$ is $\#\mathcal{N}_{n-2}+\#\mathcal{M}_{n-2}$.

Similarly, locally \GL-like formulas that are also direct can only belong to the basic locally \GL-like class $A\rhd \bot$, and there are $\#\mathcal{F}_{n-1}-\#\mathcal{N}_{n-1}$ such formulas. Next, the only direct formulas of the form $(G\rightarrow\bot)\rhd\bot$ are $(N\rightarrow\bot)\rhd\bot$ for some negative $N$, and there are $\#\mathcal{N}_{n-2}$ of them.

In fact, the only case left to do is formulas that are both locally \GL-like and negative; let us denote those formulas by $P$. These are either $\bot$ or formed as $A\rightarrow N$, where both $A$ and $N$ are locally \GL-like, so the production rules are
$$ P \,::=\, \bot \ | \ (L\setminus P) \rightarrow P. $$
With the choice of parameters
$$ p'=0,\ q'=0,\ q''=1,\ U=\{\bot\},\ V=L $$
Theorem~\ref{theorem01}(b) gives
\begin{align*}
\lambda_P & = \frac{2r}{1-\lambda_L + \sqrt{(1-\lambda_L)^2 + 4r}} = 0.13943\ldots, \\
\mu_P & = \frac{\lambda_P \mu_L}{1-\lambda_L+2\lambda_P} = 0.11865\ldots.
\end{align*}

The only nonempty triple intersection consists of formulas that are locally \GL-like, direct, and of the form $(G\rightarrow\bot)\rhd\bot$. We have already observed that such formulas can only be $(N\rightarrow\bot)\rhd\bot$ and their number is $\#\mathcal{N}_{n-2}$. We will need to add this number back, due to undercounting.

Observe that we can write
$$ G \,::=\, Q \ | \ (G\rightarrow\bot)\rhd\bot, $$
for some class of formulas $Q$ such that the above two productions yield disjoint sets of formulas of complexity $n$. We are ready to apply the inclusion-exclusion principle:
\begin{align*}
\#\mathcal{Q}_n & = \#\mathcal{L}_n + \#\mathcal{N}_n + \#\mathcal{M}_n - (\#\mathcal{N}_{n-2}+\mathcal{M}_{n-2}) \\
& \quad - (\#\mathcal{F}_{n-1}-\#\mathcal{N}_{n-1}) - \#\mathcal{N}_{n-2} - \#\mathcal{P}_n + \#\mathcal{N}_{n-2}
\end{align*}
for $n\geq 2$, so by repeated applications of Theorem~\ref{theorem01}, we easily derive
\begin{align*}
\mu_Q & = \mu_L + \mu_N + \mu_M - (r^2\mu_N+r^2\mu_M) - (r\mu_F-r\mu_N) - \mu_P \\
& = 0.92305\ldots.
\end{align*}
Finally, Theorem~\ref{theorem01}(d) with
$$ p'=0,\ q'=1,\ U=Q $$
concludes
$$ \mu_G = \frac{\mu_Q}{1-r^2} = 0.93771\ldots. $$
This completes the proof of Part (a). The ``counting procedure'' for all of the mentioned classes of \ILC\ formulas can be summarized in the following table.
\begin{center}
{\small\begin{tabular}{|c|l|l|r|}
\hline
letter & class & grammar rules & $\mu$ \\ \hline
$F$ & all \ILC\ formulas & $\bot\;|\;F\rightarrow F\;|\;F\rhd F$ & $1$ \\
$N$ & negative & $\bot\;|\;A\rightarrow N$ & $\approx0.14$ \\
$A$ & affirmative & $N\rightarrow F\;|\;F\rightarrow A\;|\;F\rhd F$ & $\approx0.86$ \\
$M$ & direct & $A\rhd N$ & $\approx0.14$ \\
$X$ & cross & $M\rightarrow\bot$ & $\approx0.02$ \\
$B$ & basic locally \GL-like & $\bot\rhd F\;|\;F\rhd A\;|\;A\rhd\bot\;|\;N\rhd X\;|\;X\rhd\bot$ & $\approx0.45$ \\
$L$ & locally \GL-like & $\bot\;|\;B\;|\;L\rightarrow L$ & $\approx0.88$ \\
$P$ & negative locally \GL-like & both $L$ and $N$ & $\approx0.12$ \\
$G$ & \GL-like & $L\;|\;N\;|\;M\;|\;(G\rightarrow\bot)\rhd\bot$ & $\approx0.94$ \\ \hline
\end{tabular}}
\end{center}

The exact numerical expression for $\mu_G$ can be obtained as an output of Mathematica \cite{Mathematica} using the command Simplify,
\begin{align*}
\mu_G & = {\scriptsize\textrm{$\frac{2}{63} \left(\!\rule{0mm}{8mm}-4-\frac{67}{34} \left(-17+3 \sqrt{17}\right)-\frac{2}{17} \left(-1683+169 \sqrt{17}\right)\sqrt{\frac{2}{-61+23 \sqrt{17}}}\right.$}} \\[-1mm]
& {\scriptsize\textrm{$\left.+\frac{\left(-1683+169
\sqrt{17}\right) \left(-16-\sqrt{-122+46 \sqrt{17}}+\sqrt{646+46 \sqrt{17}+32 \sqrt{-122+46 \sqrt{17}}}\right)}{17 \sqrt{-10710+6026 \sqrt{17}+368
\sqrt{34 \left(-61+23 \sqrt{17}\right)}-976 \sqrt{-122+46 \sqrt{17}}}}\!\right)$}}.
\end{align*}
We envisage no practical use of this radical expression, but we included it in order to emphasize that our techniques are exact.

\smallskip
(b) Once we add $\Box$ to the alphabet, the basic parameters change to $p=1$, $q=2$, $s=1$. The calculation is very similar; we only provide the table of approximate shares.
\begin{center}
{\small\begin{tabular}{|c|l|l|r|}
\hline
letter & class & grammar rules & $\mu$ \\ \hline
$F$ & all \ILC\ formulas & $\bot\;|\;\Box F\;|\;F\rightarrow F\;|\;F\rhd F$ & $1$ \\
$N$ & negative & $\bot\;|\;A\rightarrow N$ & $\approx0.11$ \\
$A$ & affirmative & $\Box F\;|\;N\rightarrow F\;|\;F\rightarrow A\;|\;F\rhd F$ & $\approx0.89$ \\
$M$ & direct & $\Box N\;|\;A\rhd N$ & $\approx0.12$ \\
$X$ & cross & $M\rightarrow\bot$ & $\approx0.01$ \\
$B$ & basic locally \GL-like & $\Box N\;|\;\Box M\;|$ &\\ & & $\bot\rhd F\;|\;F\rhd A\;|\;A\rhd\bot\;|\;N\rhd X\;|\;X\rhd\bot$ & $\approx0.44$ \\
$L$ & locally \GL-like & $\bot\;|\;B\;|\;L\rightarrow L$ & $\approx0.74$ \\
$P$ & negative locally \GL-like & both $L$ and $N$ & $\approx0.08$ \\
$G$ & \GL-like & $L\;|\;N\;|\;M\;|\;\Box G\;|\;(G\rightarrow\bot)\rhd\bot$ & $\approx0.87$ \\ \hline
\end{tabular}}
\end{center}

\smallskip
(c) If we also allow $\top$ as a primitive symbol, then we have $p=1$, $q=2$, $s=2$, and the table becomes:
\begin{center}
{\small\begin{tabular}{|c|l|l|r|}
\hline
letter & class & grammar rules & $\mu$ \\ \hline
$F$ & all \ILC\ formulas & $\bot\;|\;\top\;|\;\Box F\;|\;F\rightarrow F\;|\;F\rhd F$ & $1$ \\
$N$ & negative & $\bot\;|\;A\rightarrow N$ & $\approx0.07$ \\
$A$ & affirmative & $\top\;|\;\Box F\;|\;N\rightarrow F\;|\;F\rightarrow A\;|\;F\rhd F$ & $\approx0.93$ \\
$M$ & direct & $\Box N\;|\;A\rhd N$ & $\approx0.08$ \\
$X$ & cross & $M\rightarrow\bot$ & $\approx0.004$ \\
$B$ & basic locally \GL-like & $\Box N\;|\;\Box M\;|$ &\\ & & $\bot\rhd F\;|\;F\rhd A\;|\;A\rhd\bot\;|\;N\rhd X\;|\;X\rhd\bot$ & $\approx0.45$ \\
$L$ & locally \GL-like & $\bot\;|\;\top\;|\;B\;|\;F\rightarrow\top\;|\;L\rightarrow L$ & $\approx0.82$ \\
$P$ & negative locally \GL-like & both $L$ and $N$ & $\approx0.06$ \\
$G$ & \GL-like & $L\;|\;N\;|\;M\;|\;\Box G\;|\;(G\rightarrow\bot)\rhd\bot\;|\;\top\rightarrow G$ & $\approx0.91$ \\ \hline
\end{tabular}}
\end{center}

\end{proof}

\smallskip
\section{Acknowledgements}
The authors wish to thank Joost J. Joosten for stating a question (in private correspondence) that inspired our efforts and for many useful remarks that improved the article. We are also grateful to Neil J. A. Sloane for correcting the reference to sequence A055113. Finally, we need to thank the anonymous referee for suggesting structural reorganization of the presented material.


\smallskip


\begin{thebibliography}{}
\bibitem{GLtrace}
  S.~N.~Artemov, B.~Silver, \emph{Arithmetically complete modal theories},
  Six papers in logic (AMS translations), American Mathematical Society, 39--54 (1987)
\bibitem{BenWill}
  E.~A.~Bender, S.~G.~Williamson, \emph{Foundations of Combinatorics with Applications}, Dover (2006)
\bibitem{Bou}
  F.~Bou, J.~J.~Joosten, \emph{The Closed Fragment of IL is PSPACE Hard},
  Electronic Notes in Theoretical Computer Science, {\bf 278}, 47--54 (2011)
\bibitem{MC}
  V.~\v{C}a\v{c}i\'{c}, M.~Vukovi\'{c}, \emph{A note on normal forms for the closed fragment of system $IL$},
  Math. Commun. {\bf 17}, 195--204 (2012),\\ also available at \hbox{\verb+http://hrcak.srce.hr/file/123527+}
\bibitem{FO}
  P.~Flajolet, A.~Odlyzko, \emph{Singularity Analysis of Generating Functions}, SIAM J. Disc. Math., {\bf 3}, 216--240 (1990)
\bibitem{FS}
  P.~Flajolet, R.~Sedgewick, \emph{Analytic Combinatorics}, Cambridge University Press, 2009.
\bibitem{Goris}
  E.~Goris, J.~J.~Joosten, \emph{A new principle in the interpretability logic of all reasonable arithmetical theories},
  Logic Journal of IGPL, {\bf 19}, 1--17 (2011)
\bibitem{ILFnormalform}
   P.~H\'{a}jek, V.~\v{S}vejdar, \emph{A note on the normal form of closed formulas of interpretability logic},
   Studia Logica, {\bf 50}, 25--28 (1991)
\bibitem{OEIS}
  The On-Line Encyclopedia of Integer Sequences\textregistered (OEIS\textregistered), \hbox{\verb+https://oeis.org/+}
\bibitem{Rudin}
  W.~Rudin, \emph{Principles of Mathematical Analysis}, Third Edition, McGraw-Hill, 1976.
\bibitem{overview}
  A.~Visser, \emph{An overview of interpretability logic}, Logic Group Preprint Series, {\bf 174} (1997)
\bibitem{GF}
  H.~S.~Wilf, \emph{generatingfunctionology}, A K Peters/CRC Press, 3rd edition (2005)
\bibitem{Mathematica}
  Wolfram Research, Inc., \emph{Mathematica}, Version 9.0, Champaign, IL (2012)
\end{thebibliography}
\end{document}